\documentclass{amsart}

\usepackage[english]{babel}
\usepackage{amssymb,amsfonts,amsxtra,amsmath}
\usepackage{dsfont,mathrsfs}
\usepackage{tikz}
\renewcommand{\mathcal}{\mathscr}
\usepackage[all,cmtip]{xy}
\xyoption{dvips}
\usepackage{url}
\usepackage{stmaryrd}
\usepackage[colorlinks,plainpages,backref]{hyperref}
\usepackage[neveradjust]{paralist}
\usepackage{tikz}

\theoremstyle{definition}
\newtheorem{ntn}{Notation}
\newtheorem{dfn}[ntn]{Definition}
\theoremstyle{plain}
\newtheorem{lem}[ntn]{Lemma}
\newtheorem{prp}[ntn]{Proposition}
\newtheorem{thm}[ntn]{Theorem}

\theoremstyle{remark}

\newtheorem{que}[ntn]{Question}
\newtheorem{rmk}[ntn]{Remark}
\numberwithin{equation}{section}

\newcommand{\ideal}[1]{{\left\langle#1\right\rangle}}
\newcommand{\into}{\hookrightarrow}
\newcommand{\onto}{\twoheadrightarrow}
\newcommand{\p}{\partial}

\newcommand{\ol}{\overline}

\newcommand{\wt}{\widetilde}

\renewcommand{\AA}{\mathfrak{A}}
\renewcommand{\aa}{\mathfrak{a}}

\newcommand{\CC}{\mathds{C}}
\newcommand{\C}{\mathcal{C}}
\newcommand{\dd}{\mathfrak{d}}
\newcommand{\DD}{\mathfrak{D}}

\newcommand{\J}{\mathcal{J}}

\newcommand{\fl}{\mathfrak{l}}
\newcommand{\M}{\mathcal{M}}

\renewcommand{\O}{\mathcal{O}}

\newcommand{\QQ}{\mathds{Q}}

\renewcommand{\ss}{\mathfrak{s}}
\renewcommand{\sl}{\mathfrak{sl}}
\newcommand{\gl}{\mathfrak{gl}}

\newcommand{\mm}{\mathfrak{m}}

\DeclareMathOperator{\Ann}{Ann}

\DeclareMathOperator{\Der}{Der}

\begin{document}

\title[Free divisors with normal crossings]{Quasihomogeneous free divisors with only normal crossings in codimension one}

\author[X.~Liao]{Xia Liao}
\address{X.~Liao\\
KIAS\\
85 Hoegiro, Dongdaemun-gu\\
Seoul 130-722\\
Republic of Korea}
\email{\href{mailto:liao@kias.re.kr}{liao@kias.re.kr}}

\author[M.~Schulze]{Mathias Schulze}
\address{M.~Schulze\\
Department of Mathematics\\
University of Kaiserslautern\\
67663 Kaiserslautern\\
Germany}
\email{\href{mailto:mschulze@mathematik.uni-kl.de}{mschulze@mathematik.uni-kl.de}}

\thanks{The research leading to these results has received funding from the People Programme (Marie Curie Actions) of the European Union's Seventh Framework Programme (FP7/2007-2013) under REA grant agreement n\textsuperscript{o} PCIG12-GA-2012-334355.}




\begin{abstract}
We prove that any divisor as in the title must be normal crossing.
\end{abstract}

\maketitle
\tableofcontents

\section*{Introduction}

Free divisors are maximally singular reduced complex hypersurfaces defined locally as a determinant of a so-called Saito matrix (see \cite[\S1]{Sai80}).
They occur classically as discriminants in the base space of versal deformations of isolated complete intersection singularities, where the Saito matrix represents a basis of liftable vector fields (see \cite[(3.16) Cor.]{Loo84}).
While there is an abundance of specific situations in which certain hypersurfaces are free divisors, it is unclear under what geometric conditions a general hypersurface can be a free divisor.
Since free divisors have purely one-codimensional singular locus, it is natural to ask what singularities can occur in a free divisor in codimension one.
The simplest case to consider is that of free divisors with only normal crossings in codimension one.
However, besides normal crossing divisors themselves, no such objects are known.
This raises the following

\begin{que}[{Faber~\cite[Que.~5]{Fab15}}]\label{41}
Is a free divisor normal crossing if it is normal crossing in codimension one?
\end{que}

In a particular case this question was answered positively.

\begin{thm}[{Granger--Schulze~\cite[Thm.~1.8]{GS14}}]\label{39}
A free divisor germ with smooth normalization is normal crossing if and only if it is normal crossing in codimension one.
\end{thm}

For free hyperplane arrangements the statement can be verified directly by induction on the codimension using the minimal number of generators of logarithmic derivations along generic arrangements determined by Yuzvinsky~\cite{Yuz91}.
As our main result, we shall replace the condition on the normalization in Theorem~\ref{39} by quasihomogeneity.

\begin{thm}\label{0}
A quasihomogeneous free divisor germ is normal crossing if and only if it is normal crossing in codimension one.
\end{thm}

A classical example of a hypersurface singularity that is quasihomogeneous and normal crossing in codimension one is the Whitney umbrella (see Figure~\ref{59}).
As a consequence of Theorem~\ref{0} it can not be a free divisor.

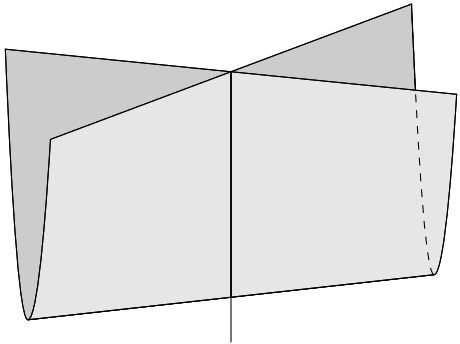
\begin{figure}[h]
\begin{tikzpicture}[scale=0.3]
\draw[fill=black!20,] (9,1) parabola (8,13) -- (0,10) -- (0,0) --(9,1);
\draw[fill=black!20] (-9,-1) parabola (-10,11) -- (0,10) -- (0,0) --(-9,-1);
\draw (9,1) parabola (8,13) -- (0,10) -- (0,0) --(9,1);
\draw[] (-9,-1) parabola (-10,11) -- (0,10) -- (0,0) --(-9,-1);
\draw[fill=black!10] (9,1) parabola (10,9) -- (0,10) -- (0,0) --(9,1);
\draw[fill=black!10] (-9,-1) parabola (-8,7) -- (0,10) -- (0,0) --(-9,-1);
\draw[dashed] (9,1) parabola (8,13);
\draw (9,1) parabola (10,9) -- (0,10) -- (0,0) --(9,1);
\draw (-9,-1) parabola (-8,7) -- (0,10) -- (0,0) --(-9,-1);
\draw (0,0) -- (0,-2);
\end{tikzpicture}

\caption{A real picture of the Whitney umbrella defined by $x^2-y^2z$.}\label{59}
\end{figure}

Although we do not know how to remove it, the additional homogeneity hypothesis is supported by a result of Granger--Schulze~\cite[Thm.~1.1]{GS14} and Faber~\cite[Rmk.~11]{Fab15}: 
free divisor germs which are normal crossing in codimension one are Euler homogeneous.
By definition a reduced hypersurface singularity is Euler homogeneous if it admits a defining equation that is fixed by a derivation.
For the stronger notion of quasihomogeneity a diagonalizable derivation with positive eigenvalues is needed (see \eqref{14}).
The two notions of homogeneity coincide in case of isolated hypersurface singularities (see \cite{Sai71}).
Free divisors with isolated singularities are plane curves and Question~\ref{41} becomes vacuous in this case.

There are many examples of highly reducible free divisors such as free hyperplane arrangements (see for example \cite[\S4]{OT92}) or free discriminants in prehomogeneous vector spaces (see \cite{BM06,GMS11}).
On the other hand, it is empirically known that irreducible free divisors are rare objects (see \cite{ST14}).
Our result supports this observation by deducing reducibility from local reducibility in codimension one (under the stated additional hypotheses).

The question of non-existence of negative degree derivations on quasihomogeneous isolated complete intersection singularities occurs in conjectures of S.~Halperin (see \cite[\S39]{FHT01}, \cite{Hau02}) and J.~Wahl (see \cite[Conj.~1.4]{Wah83}, \cite[p.~475, Thm.]{Ale85}).
While quasihomogeneous free divisors may admit negative degree logarithmic derivations, a related extreme plays a role in our present context:
For a quasihomogeneous free divisor germ $D$, we consider the (finite dimensional) complex Lie algebra $\dd$ (see \eqref{62}) generated by all weighted homogeneous logarithmic derivations along $D$ of non-positive degree.
As a crucial ingredient of our approach, we show that, after reduction to the case where $D$  is irreducible and unsuspended, the hypotheses of Theorem~\ref{0} imply that $\dd$ generates the module of all logarithmic derivations along $D$ (see Proposition~\ref{20}).
We achieve the proof of the theorem by studying the Lie algebra representation of $\dd$ on the space of lowest weight variables (see \S\ref{17}).

We shall now give a more detailed summary of our approach to proving Theorem~\ref{0}:
The statement can be reduced to the irreducible unsuspended case (see Proposition~\ref{42}, Remark~\ref{45} and Lemmas~\ref{47} and \ref{49}).

In \S\ref{22}, we consider a free divisor germ $D$ which is normal crossing in codimension one.
Combining results from Granger--Schulze~\cite{GS14} and Mond--Pellikaan~\cite{MP89}, we prove that the ideal of submaximal minors of a Saito matrix coincides with the Jacobian ideal (see Proposition~\ref{34}).

In \S\ref{27}, we assume in addition that $D$ is quasihomogeneous.
Then a degree argument using Saito's criterion shows that there is a weighted homogeneous basis of logarithmic derivations of non-positive degree (see Proposition~\ref{20}).

In \S\ref{51} and \S\ref{17}, we consider the representation $\dd'$ of the Lie algebra $\dd$ on the vector space $\mm'$ of lowest weight variables.
This representation is visible as a block in a Saito matrix (see \eqref{30}).
In case $\dd'$ is solvable, Lie's theorem and Saito's criterion enforce smoothness of $D$ (see Proposition~\ref{31}).
In case $\dd'$ is not solvable, we describe a change of coordinates and weights that makes the dimension of $\mm'$ drop.
Iterating this process, we arrive at a solvable $\dd'$ returning to the previous case (see Proposition~\ref{33}).

\subsection*{Acknowledgments}

The first author would like to thank the Department of Mathematics at
the University of Kaiserslautern for providing a pleasant working
environment during his postdoctoral stay.

\section{Normal crossings in codimension one}\label{22}

Let $X:=(\CC^{n+1},0)$ be the germ of complex $(n+1)$-space and let $D$ be the germ of a reduced hypersurface in $X$ with ideal generated by a function $f\in\O_X$. 
We shall use the term divisor as a synonym for a reduced hypersurface.
The Jacobian ideal of $D$ is defined as the ideal
\[
\J_D:=\ideal{\frac{\p f}{\p x_1},\dots,\frac{\p f}{\p x_{n+1}}}_{\O_D}
\]
generated by the partial derivatives of $f$ with respect to coordinates $x_1,\dots,x_{n+1}$ of $X$. 
Recall that there are two exact sequences of $\O_X$-modules:
\begin{gather}
\xymatrix{
0 \ar[r] &  \Der_X(-\log D) \ar[r] & \Der_\CC(\O_X) \ar[r] & \J_D \ar[r] &  0,
}\label{13}\\
\xymatrix{
0 \ar[r] &  \Omega^p_X \ar[r] &  \Omega^p_X(\log D) \ar[r]^-{\rho_D^p} & \omega^{p-1}_D \ar[r] &  0.
}\label{12}
\end{gather}
In \eqref{12}, $\omega^{p-1}_D$ is the image of $\Omega_X^p(\log D)$ in $\M_D\otimes_{\O_D}\Omega^{p-1}_D$ under the Saito's logarithmic residue map $\rho_D^p$ (see \cite[\S2]{Sai80}), where $\M_D$ is the ring of meromorphic functions on $D$. 
Aleksandrov~\cite[\S4, Thm.]{Ale90} proved that $\omega^p_D$ agrees with the module of $p$th regular differential forms on $D$ (see \cite{Bar78}). 
For further properties of the logarithmic residue map, we refer to \cite{GS14}.

Let us assume that $D$ is normal crossing in codimension one which means that, outside of an analytic subset codimension at least two, a representative of $D$ is locally isomorphic to a normal crossing divisor.
Then we know by \cite[Thm ~1.1]{GS14}, that $\omega^0_{D}$ can be identified with the ring of weakly holomorphic functions on $D$. 
Denoting by $\pi\colon\tilde D\to D$ the normalization map,
this ring, in turn, can be identified with $\pi_*\O_{\tilde D}$ where $\O_{\tilde D}=\wt{\O_D}$ is the integral closure of $\O_D$ in $\M_D$. 
Combining these facts, turns \eqref{12} for $p=1$ into a short exact sequence
\begin{equation}\label{9}
\xymatrix{
0 \ar[r] &  \Omega^1_X \ar[r] &  \Omega^1_X(\log D) \ar[r]^-{\rho_D^1} &  \pi_*\O_{\tilde D} \ar[r] &  0
}
\end{equation}
whenever $D$ is normal crossing in codimension one. 
The dual of the inclusion map in \eqref{9} is the inclusion map in \eqref{13}.
Note $\pi_*\O_{\tilde D}$ is finite as $\O_D$-module and hence also as $\O_X$-module. 

Let us further assume that $D$ is a free divisor germ. 
This means that $\Omega^1_X(\log D)$ is a locally free sheaf of rank $n+1=\dim X$. 
Since \eqref{9} is a free $\O_X$-resolution of $\pi_*\O_{\tilde D}$ of length one, $\pi_*\O_{\tilde D}$ is a Cohen--Macaulay $\O_X$-module by the Auslander--Buchsbaum formula. 
Since $\pi$ is finite it follows that $\tilde D$ is Cohen--Macaulay.
Then \cite[Proof of Thm.~3.4]{MP89} applies to show that the first Fitting ideal of $\pi_*\O_{\tilde D}$ equals the preimage $\C'_D:=\Ann_{\O_X}(\pi_*\O_{\tilde D}/\O_D)$ in $\O_X$ of the conductor ideal
\[
\C_D:=\Ann_{\O_D}(\pi_*\O_{\tilde D}/\O_D).
\] 
By \cite[Thm.~1.6]{GS14} (or \cite[Proof of Thm.~3.4]{MP89}), $\C_D=\J_D$ and hence $\C'_D$ equals the preimage $\J'_D:=\ideal{f,\frac{\p f}{\p x_1},\dots,\frac{\p f}{\p x_{n+1}}}_{\O_X}$ in $\O_X$ of the Jacobian ideal $\J_D$.

\begin{rmk}\label{48}
Suppose that $D$ is a free divisor germ.
For dual bases $\delta_1,\dots,\delta_{n+1}$ and $\omega_1,\dots,\omega_{n+1}$ of $\Der_X(-\log D)$ and $\Omega_X^1(\log D)$ there are expansions 
\begin{equation}
\delta_j=\sum_{i=1}^{n+1}a_{i,j}\frac{\p}{\p x_i},\quad dx_i=\sum_{j=1}^{n+1}\lambda_{j,i}\omega_i,\quad a_{i,j}=\lambda_{j,i}\in\O_X.\label{44}
\end{equation}
A matrix $A:=(a_{i,j})$ arising is this way is called a Saito matrix and represents the inclusion map in \eqref{13}.
The inclusion in \eqref{9} is represented by its transpose $\Lambda:=(\lambda_{i,j})=A^t$.
Saito's freeness criterion (see \cite[(1.8) Thm.~ii)]{Sai80}) states that
\begin{equation}\label{56}
f:=\det(A)=\det(\Lambda)\in\O_X
\end{equation}
generates the ideal of $D$.
\end{rmk}

The preceding discussion proves the following

\begin{prp}\label{34}
Let $D$ be a free divisor germ in $X$ which is normal crossing in codimension one. 
Then the submaximal minors of any Saito matrix $A$ generate the preimage $\J'_D$ in $\O_{X}$ of the Jacobian ideal $\J_D$ of $D$.\qed
\end{prp}

The following result justifies that we may assume that $D$ is irreducible.

\begin{prp}\label{42}
The combination of freeness and normal crossing in codimension one descends to all irreducible components of a divisor germ.
\end{prp}

\begin{proof}
For any free divisor germ $D$, the normal crossing hypothesis is equivalent to the ideal
\begin{equation}\label{61}
\J_f:=\ideal{\frac{\p f}{\p x_1},\dots,\frac{\p f}{\p x_{n+1}}}_{\O_X}
\end{equation}
generated by the partial derivatives of $f$ being radical by \cite[Thm.~1.6, Rmk.~1.7]{GS14}.
By \cite[\S 2.2.(i)]{Fab15}, the combination of freeness and radical $\J_f$ descends to all irreducible components of $D$.
\end{proof}

\section{Suspension of divisors}\label{29}

For a degree argument in Proposition~\ref{20}, we shall need the following 

\begin{lem}\label{43}
Let $D$ be an irreducible free divisor germ satisfying 
\begin{equation}\label{18}
\Der_X(-\log D)\subset\mm_X\Der_\CC(\O_X)
\end{equation}
and let $\Lambda$ be a transposed Saito matrix as in Remark~\ref{48}.
Denote by $M_{i,j}$ the $(n\times n)$-minor obtained from $\Lambda$ by deleting the $i$th row and $j$th column.
Then $M_{i,j}\ne0$ for all $i,j=1,\dots,n+1$.
\end{lem}

\begin{proof}
With notation as in Remark~\ref{48}, let $g_i:=\rho_D^1(\omega_i)$ be the image of $\omega_i$ under the residue map in \eqref{9}.
By Cramer's rule (see \cite[Lem.~3.3]{MP89}), we have
\[
M_{i,j}g_k=\pm M_{k,j}g_i
\]
in $\O_{\tilde D}$ which is a domain by irreducibility of $D$.
Assuming that $M_{i,j}=0$ for some $i,j\in\{1,\dots,n+1\}$, this implies that $g_i=0$ or $M_{k,j}=0$ for all $k=1,\dots,n+1$.
In the first case, $\omega_i\in\Omega_X^1$ by the residue sequence \eqref{9}.
By duality and hypothesis~\eqref{18} this gives the contradiction $1=\ideal{\delta_i,\omega_i}\in\mm_X$.
Therefore, we must have $M_{k,j}=0$ in $\O_{\tilde D}$ and hence in $\O_D$ for all $k=1,\dots,n+1$.
In other words, $f$ divides $M_{k,j}$ in $\O_X$ for all $k=1,\dots,n+1$.
In terms of \eqref{44}, hypothesis~\eqref{18} means that $\lambda_{j,i}\in\mm_X$.
Thus, Saito's criterion (see Remark~\ref{48}) leads to the contradiction
\[
f=\det(\Lambda)=\sum_{k=1}^{n+1}\pm\lambda_{k,j}M_{k,j}\in\mm_Xf.
\]
\end{proof}

\begin{rmk}\label{45}
The following construction will serve to reduce to a case where condition~\eqref{18} is satisfied (see Lemma~\ref{47}).
By \cite[(3.6) Proof]{Sai80}, there is a product structure
\begin{equation}\label{46}
(D,X)\cong (D'\times X'',X'\times X''),
\end{equation}
where $X'':=(\CC^r,0)$ for some $r\in\{0,\dots,n\}$, such that the divisor germ $D'$ in $X':=(\CC^{n+1-r},0)$ satisfies condition~\eqref{18}.
In particular, $D$ is smooth if and only if $r=n$.
An $X'$ as above can be chosen as any (coordinate) subspace complementary to
\[
\Der_X(-\log D)(0)=T_0X''.
\]
By the cancellation lemma for analytic spaces (see \cite[Lem.~1.5]{Eph78}), all such $D'$ are isomorphic.
We say that $D$ is a suspension of $D'$, or suspended, if $r>0$.

Whenever $D$ and $D'$ are related as in \eqref{46}, we have
\begin{equation}\label{54}
\Der_X(-\log D)=\O_X\Der_{X'}(-\log D')\oplus\O_X\Der_\CC(\O_{X''}).
\end{equation}
For the dual statement see \cite[Lem.~2.2.(iv)]{CMN96}.
\end{rmk}

The following lemma summarizes properties of suspension.

\begin{lem}\label{47}
Let $D$ be a suspension of $D'$ as in \eqref{46}.
\begin{enumerate}[(a)]
\item\label{47a} Suspension defines a bijection between the sets of irreducible components of $D$ and $D'$.
\item\label{47b} The following properties hold equivalently for $D$ and $D'$: smoothness of corresponding irreducible components, being normal crossing (in codimension one) and freeness.
\end{enumerate}
\end{lem}

\begin{proof}
The statement on freeness in \eqref{47b} follows from the equality \eqref{54}.
The proof of the remaining statements is left to the reader.
\end{proof}

For any $f\in\O_X$ generating the ideal of $D$ we denote the module of (logarithmic) derivations annihilating $f$ by
\begin{equation}\label{55}
\Der_X(-\log f):=\{\delta\in\Der_\CC(\O_X)\mid\delta(f)=0\}\subset\Der_X(-\log D).
\end{equation}
It is isomorphic to the syzygy module of the ideal $\J_f$ of partial derivatives of $f$ defined in \eqref{61}.

\section{Quasihomogeneous divisors}\label{27}

From now on we assume that $D$ is quasihomogeneous.
By definition, this means that $\O_D$ is a positively graded $\CC$-algebra in the sense of Scheja--Wiebe~(see \cite[\S3 Def.]{SW73}).
In more explicit terms, there are coordinates $x_1,\dots,x_{n+1}$ and weights 
\begin{equation}\label{5}
\deg(x_i)=:w_i\in\QQ_+,\quad i=1,\dots,n+1, 
\end{equation}
such that $D$ is defined by a weighted homogeneous polynomial $f$.
We shall order the weights increasingly by
\begin{equation}\label{26}
0<w_1\le w_2\le\cdots\le w_{n+1}.
\end{equation}
and denote by
\begin{equation}\label{14}
\chi:=\sum_{i=1}^{n+1}w_ix_i\frac{\p}{\p x_i}
\end{equation}
the corresponding logarithmic Euler derivation along $D$. 
Weighted homogeneity of degree $\deg(g)$ of an element $0\ne g\in\O_X$ can be expressed by the condition that
\begin{equation}\label{15}
[\chi,g]=\chi(g)=\deg(g)g
\end{equation}
where the commutator is taken in the ring of $\CC$-linear differential operators on $\O_X$.

\begin{dfn}
By an Euler derivation on a complex space germ $X$ we mean a derivation $\chi\in\Der_\CC(\O_X)$ such that $\mm_X$ is generated by eigenvectors of $\chi$ with eigenvalues in $\QQ_+$.
\end{dfn}

Being logarithmic along $D$ means that the ideal of $D$ is $\chi$-stable and hence generated by a weighted homogeneous $f$ (see \cite[(2.4)]{SW73}).
A logarithmic Euler derivation $\chi\in\Der_\CC(\O_X)$ descends to an Euler derivation $\ol\chi\in\Der_\CC(\O_D)$ and conversely such a $\bar\chi$ lifts to a such $\chi$ (see \cite[(2.1)]{SW73}).
It is the existence of $\ol\chi$, or equivalently that of $\chi$, that defines quasihomogeneity of $D$ (see \cite[(2.2), (2.3)]{SW73}).

\begin{rmk}\label{28}
There is a decomposition as $\O_X$-modules
\begin{equation}\label{60}
\Der_X(-\log D)=\O_X\chi\oplus\Der_X(-\log f).
\end{equation}
Indeed, any $\delta\in\Der_X(-\log D)$ can be written as $\delta=a\chi+\delta-a\chi$ where $\delta-a\chi\in\Der_X(-\log f)$ for $a=\frac{\delta(f)}{\chi(f)}\in\O_X$.
Moreover, $a\chi\in\Der_X(-\log f)$ implies $a\deg(f)f=a\chi(f)=0$ and hence $a=0$ for any $a\in\O_X$.
\end{rmk}

Quasihomogeneity of $D$ descends to any irreducible component of $D$ using a fixed $\chi$.
This is a particular case of a result of Seidenberg (see \cite[Thm.~1]{Sei67}).
For later use we record also the compatibility of quasihomogeneity with suspension.

\begin{lem}\label{49}\
\begin{enumerate}[(a)]
\item\label{49a} If $D$ is quasihomogeneous, then all its irreducible components are quasihomogeneous with respect to the same weights as $D$.
\item\label{49b} If $D$ is a suspension of $D'$ as in \eqref{46} then quasihomogeneity of $D$ is equivalent to that of $D'$.
The weights for $D$ on the coordinates of $X$ consist of the weights for $D'$ on the coordinates of $X'$ and arbitrary weights on coordinates of $X''$.
\end{enumerate}
\end{lem}

\begin{proof}
To show that quasihomogeneity passes from $D$ to $D'$, one restricts the corresponding logarithmic Euler derivation $\chi$ from \eqref{14} to a coordinate subspace $X'$ as in Remark~\ref{45}.
The equation of $D’$ and this restriction $\chi'$ is then obtained by setting the coordinates 
defining $X'$ equal to $0$.
\end{proof}

The assignment of weights \eqref{5} turns $\Der_\CC(\O_X)$ into a positively graded $\O_X$-module in the sense of Scheja and Wiebe (see \cite{SW73}).
A derivation $0\ne\delta\in\Der_\CC(\O_X)$ is weighted homogeneous of degree $\deg(\delta)$, if the coefficients $a_i$ of the expansion $\delta=\sum_{i=1}^{n+1}a_i\frac{\p}{\p x_i}$ are weighted homogeneous of degree $\deg(a_i)=\deg(\delta)+w_i$ for all $i=1,\dots,n+1$. 
In other words, 
\[
\deg(\delta(g))=\deg(\delta)+\deg(g)
\]
whenever $\delta(g)\ne0$ for a weighted homogeneous polynomial $g$. 
For example, the Euler derivation $\chi$ in \eqref{14} is weighted homogeneous of degree 
\begin{equation}\label{63}
\deg(\chi)=0
\end{equation}
for any assignment of weights.
Analogously to \eqref{15}, the weighted homogeneity of $\delta$ can be expressed by
\begin{equation}\label{19}
[\chi,\delta]=\deg(\delta)\delta.
\end{equation}
Indeed, for any weighted homogeneous polynomial $g$, we have that
\begin{equation}\label{16}
[\chi,\delta](g)=\chi(\delta(g))-\delta(\chi(g))=\left(\deg(\delta(g))-\deg(g)\right)\delta(g).
\end{equation}
Due to \eqref{15} and \eqref{19}, we prefer to speak of $\chi$-weights, $\chi$-homogeneity and the $\chi$-degree 
\[
\deg(-)=\deg_\chi(-)
\]
independently of coordinates.
For notational convenience, we set $\deg(0):=-\infty$.

A $\chi$-homogeneous generator $f$ of the ideal of $D$ is uniquely determined by $\chi$ up to a constant factor as a $\chi$-homogeneous element of minimal $\chi$-degree.
Then also the module $\Der_X(-\log f)$ in \eqref{55} is uniquely determined by $\chi$.

\section{Lie algebras of logarithmic derivations}\label{51}

It is known that $\Der_X(-\log D)$ is closed under the bracket operation $[-,-]$ (see \cite[(1.5) ii)]{Sai80}).
In particular, $\chi$ acts on $\Der_X(-\log D)$ turning it into a positively graded $\O_X$-module in the sense of Scheja and Wiebe (see \cite{SW73}).
Since $f$ is weighted homogeneous the same holds true for its ideal of partials $\J_f$ and its syzygy module $\Der_X(-\log f)$.

\begin{ntn}\label{25}
For any $\CC[\chi]$-module $M$ and $e\in\QQ$, we denote by $M_e=M_{\chi=e}$ the $\CC[\chi]$-submodule generated by all $\chi$-homogeneous elements $m\in M$ of $\chi$-degree $\deg(m)=e$.
The submodules $M_{\le e}=M_{\chi\le e}$, $M_{<e}=M_{\chi<e}$, etc.~are defined analogously.
\end{ntn}

For any two $\chi$-homogeneous $\delta,\eta\in\Der_X(-\log D)$ with $[\delta,\eta]\ne0$, also their commutator $[\delta,\eta]$ is $\chi$-homogeneous of degree
\begin{equation}\label{21}
\deg([\delta,\eta])=\deg(\delta)+\deg(\eta).
\end{equation}
This follows by definition, or from \eqref{19} and the Jacobi identity
\begin{equation}\label{37}
[\chi,[\delta,\eta]]=[[\chi,\delta],\eta]+[\delta,[\chi,\eta]].
\end{equation}
Since $\dim_\CC(\O_{X,\le e})<\infty$ for any $e\in\QQ_+$ due to the positivity of weights \eqref{5},
\begin{align}\label{62}
\aa&=\aa_\chi:=\Der_X(-\log f)_{\le0}\subset\\
\nonumber\dd&=\dd_\chi:=\Der_X(-\log D)_{\le0}\subset\Der_\CC(\O_X)_{\le0}=\bigoplus_{i=1}^{n+1}\O_{X,\le w_i}\frac\p{\p x_i}
\end{align}
are finite dimensional complex Lie algebras depending on $\chi$.
By \eqref{60}, \eqref{63}, \eqref{19} and \eqref{37}, $\dd$ decomposes as
\begin{equation}\label{23}
\dd=\CC\chi\ltimes\aa
\end{equation}
and 
\begin{equation}\label{24}
\aa=\dd\cap\Der_X(-\log f)\subset\dd
\end{equation}
is the Lie subalgebra of $\dd$ of derivations annihilating $f$.
We shall be interested in the condition that $\Der_X(-\log D)$ admits $\chi$-homogeneous $\O_X$-generators of non-positive $\chi$-degree, that is, 
\begin{equation}\label{36}
\O_X\dd_\chi=\Der_X(-\log D).
\end{equation}
Let $\{\delta_2,\dots,\delta_r\}$ be a minimal set of $\chi$-homogeneous generators of $\Der_X(-\log f)$ of $\chi$-degree
\[
\deg(\delta_i)=:d_i.
\]
Then condition~\eqref{36} is equivalent to $\delta_i\in\dd$ or to $d_i\le0$ for all $i=1,\dots,n+1$.
Setting $\delta_1:=\chi$ and $d_1:=\deg(\chi)=0$,
\[
\Delta=\{\delta_1,\dots,\delta_r\}
\]
becomes a minimal set of $\chi$-homogeneous generators of $\Der_X(-\log D)$ by \eqref{60}.
If $D$ is a free divisor germ, then $r=n+1$, $\Delta$ is a $\chi$-homogeneous basis of $\Der_X(-\log D)$ and the $\chi$-degree of any entry $\lambda_{i,j}\ne0$ of the transposed Saito matrix $\Lambda$ from Remark~\ref{48} equals 
\begin{equation}\label{11}
\deg(\lambda_{i,j})=d_i + w_j.
\end{equation}
The defining function $f$ of $D$ given by Saito's criterion as in \eqref{56} is then $\chi$-homogeneous of $\chi$-degree
\begin{equation}\label{57}
\deg(f)=\sum_{k=1}^{n+1}d_k+w_k.
\end{equation}

\begin{prp}\label{20}
Let $D$ be an irreducible unsuspended quasihomogeneous free divisor germ which is normal crossing in codimension one, and let $\chi$ be any logarithmic Euler derivation along $D$.
Then the module $\Der_X(-\log D)$ admits a $\chi$-homogeneous basis whose elements have non-positive $\chi$-degree.
In other words, $D$ satisfies condition~\eqref{36} for any choice of logarithmic Euler derivation $\chi$ along $D$.
\end{prp}

\begin{proof}
By hypothesis of being unsuspended $D$ satisfies condition~\eqref{18} and hence $\deg(M_{i,j})\ne-\infty$ by Lemma~\ref{43}.
Using equalities~\eqref{11} and \eqref{57}, we obtain
\begin{align}\label{2}
\deg(M_{i,j}) 
&=\sum_{k=1}^{n+1}(d_k+w_k)-(d_i+w_j) \\
&=\deg(f)-w_j-d_i\nonumber\\
&=\deg\left(\frac{\partial f}{\partial x_j}\right)-d_i.\nonumber
\end{align}
By Proposition~\ref{34}, there is an equality of ideals 
\[
\ideal{M_{i,j}\mid i,j=1,\dots,n+1}_{\O_X}=\ideal{\frac{\partial f}{\partial x_k}\mid k=1,\dots,n+1}_{\O_X}.
\]
We shall use that the least $\chi$-degree of a system of generators is an invariant of this ideal.
Setting $k=n+1$ minimizes $\deg(\frac{\partial f}{\partial x_k})=\deg(f)-w_k$ due to the increasing order of weights in \eqref{26}.
It follows that
\begin{equation}\label{58}
\deg(M_{i,j})\ge\deg\left(\frac{\partial f}{\partial x_{n+1}}\right)
\end{equation}
for all $i,j=1,\dots,n+1$.
Combining \eqref{2} and \eqref{58} for $j=n+1$ now yields
\[
\deg\left(\frac{\partial f}{\partial x_{n+1}}\right)-d_i
=\deg(M_{i,n+1})\ge\deg\left(\frac{\partial f}{\partial x_{n+1}}\right)
\]
and hence $d_i\le 0$ for all $i=1,\dots,n+1$ as claimed.
\end{proof}

\section{Representation on lowest weight variables}\label{17}

For notational convenience, we shall abbreviate 
\[
\mm:=\mm_X,\quad\DD:=\Der_X(-\log D),\quad\AA:=\Der_X(-\log f).
\]
Then $\dd=\DD_{\le0}$ and $\aa=\AA_{\le0}$ (see \eqref{62}).
In particular, $\dd_0=\DD_0$ and $\aa_0=\AA_0$.

From now on we assume that $D$ is unsuspended, that is, \eqref{18} holds true.
Then $\mm$ and hence, by the Leibniz rule, also $\mm^2$ is a $\DD$-module.
The canonical $\O_X$-module homomorphism
\[
\pi\colon\mm\onto\mm/\mm^2=:\bar\mm
\]
becomes a $\DD$-module homomorphism.
Splittings of $\pi$ can be seen as intrinsic versions of coordinate systems. 
For any $e\in\QQ_+$, $\pi$ induces a $\dd_0$-module homomorphism (see Notation~\ref{25})
\begin{equation}\label{64}
\pi_e\colon\mm_e\onto\bar\mm_e.
\end{equation}
Splittings of $\pi$ compatible with the $\pi_e$ can be seen as $\chi$-homogeneous coordinate systems.
For any $e\in\QQ_+$, $\mm_e\subset\mm_{\le e}$ and $\bar\mm_e\subset\bar\mm_{\le e}$ are homomorphisms of modules over $\dd_0\subset\dd$.
Note that $\dim_\CC(\mm_{\le e})<\infty$ due to the positivity of weights \eqref{5}.

For any $e\in\QQ_+$ with $w_1\le e$, there is a commutative diagram of $\dd$-modules
\begin{equation}\label{35}
\xymatrix{
\mm\ar@{->>}[r]^-\pi & \bar\mm\\
\mm_{\le e}\ar@{->>}[r]^-{\pi_{\le e}}\ar@{^(->}[u] & \bar\mm_{\le e}\ar@{^(->}[u]\\
\mm_{w_1}\ar@{^(->}[u]\ar@{^(->>}[r]^-{\pi_{w_1}} & \bar\mm_{w_1}\ar@{^(->}[u]
}
\end{equation}
where $\pi_{w_1}$ is a bijection that we consider as an equality and
\[
\mm'=\mm'_\chi:=\mm_{w_1}=\mm_{\le w_1}=\bar\mm_{w_1}=\bar\mm_{\le w_1}
\]
is the $\CC$-vector space of lowest weight variables.
After reordering coordinates, we may assume that $x_1,\dots,x_k$ is a $\CC$-basis of $\mm'$.
In particular,
\begin{equation}\label{65}
w_1=\dots=w_k. 
\end{equation}

\begin{ntn}\label{3}
For any Lie subalgebra $\fl\subset\dd_\chi$ we denote by $\bar\fl=\bar\fl_\chi$ and  $\fl'=\fl'_\chi$ the image of $\fl$ in $\gl_\CC(\bar\mm)$ and $\gl_\CC(\mm')$ respectively.
Note that $\fl'=\fl_0'$ by definition of $\mm'$.
\end{ntn}

From now on we assume that condition~\eqref{36} is satisfied for some $\chi$, that is,
\[
\delta_2,\dots,\delta_r\in\aa
\]
for any minimal set of $\chi$-homogeneous $\O_X$-generators $\{\delta_2,\dots,\delta_r\}$ of $\AA$ as in \S\ref{51}.

Any $\delta\in\dd$ can be written as $\delta=\sum_{i=1}^{n+1}a_i\frac{\p}{\p x_i}$ where $a_1,\dots,a_k$ must be linear functions of $x_1,\dots,x_k$ for degree reasons, and it maps to 
\begin{equation}\label{53}
\delta'=\sum_{i=1}^{k}a_i\frac{\p}{\p x_i}\in\dd'.
\end{equation}
If $\delta\in(\mm\DD)_{\le0}$ then $\delta\in\mm^2\Der_\CC(\O_X)$ due to hypothesis~\eqref{18} of $D$ being unsuspended.
But then the linear coefficients of $\delta'$ and hence $\delta'$ itself must be zero.
Since $\delta_1=\chi,\delta_2\dots,\delta_r\in\dd$ generate $\DD$ as $\O_X$-module, the inclusion in the diagram of $\CC$-linear maps
\[
\xymatrixcolsep{0em}\xymatrix{
\dd\ar@{>>}[rr]\ar[dr] && \DD/\mm\DD\\
& \dd/(\mm\DD)_{\le0}\ar@{^(->}[ru]
}
\]
is an equality.
By the preceding arguments, $\dd\onto\dd'$ factors through the resulting surjection 
\[
\dd\onto\DD/\mm\DD=\dd/(\mm\DD)_{\le0}.
\]
Arguing analogously for $\aa\onto\aa'$, we obtain
\[
\xymatrix{
\dd\ar@{>>}[r]\ar@{>>}[d] & \dd'\\
\DD/\mm\DD\ar@{=}[r] & \dd/(\mm\DD)_{\le0},\ar@{>>}[u]
}\quad
\xymatrix{
\aa\ar@{>>}[r]\ar@{>>}[d] & \aa'\\
\AA/\mm\AA\ar@{=}[r] & \aa/(\mm\AA)_{\le0}.\ar@{>>}[u]
}
\]
In explicit terms this means that 
\begin{equation}\label{52}
\dd'=\ideal{\delta'_1,\dots,\delta'_r}_\CC,\quad
\aa'=\ideal{\delta'_2,\dots,\delta'_r}_\CC.
\end{equation}

In case $D$ is a free divisor germ, \eqref{53} and \eqref{52} show that $\dd'$ and $\aa'$ are represented as blocks of a Saito matrix with linear entries.
Indeed, the transposed Saito matrix from Remark~\ref{48} can be visualized as
\begin{equation}\label{30}
\Lambda=
\left[
\begin{array}{c|c}
\dd' & * 
\end{array}
\right]=
\left[
\begin{array}{c|c}
w_1x_1\cdots w_kx_k & * \\
\hline
\aa' & * 
\end{array}
\right]
\end{equation}
where by abuse of notation $\dd'$ and $\aa'$ represent the coefficient matrices of $\delta_1,\dots,\delta_{n+1}$ and $\delta_2,\dots,\delta_{n+1}$ with respect to $\frac\p{\p x_1},\dots,\frac\p{\p x_k}$.
 
\begin{prp}\label{31}
Let $D$ be a quasihomogeneous free divisor germ, and let $\chi$ be a logarithmic Euler derivation along $D$.
Suppose that $D$ satisfies both condition~\eqref{18} of being unsuspended and condition~\eqref{36} of admitting non-positive $\chi$-degree generators of logarithmic derivations.
If $\aa'=\aa'_\chi$ is solvable then $D$ has a smooth irreducible component $D'$.
\end{prp}

\begin{proof}
By Lie's theorem, we may assume that $x_1$ is a common eigenvector for $\aa'$.
This means that the first column of $\Lambda$ in \eqref{30} consists of constant multiples of $x_1$ and hence $x_1$ divides $\det(\Lambda)$.
By Saito's criterion (see Remark~\ref{48}), $f=\det(\Lambda)$ defines $D$ and hence $\{x_1=0\}$ is a smooth irreducible component of $D$.
\end{proof}

\begin{lem}\label{32}
Let $\ss'\subset\dd'$ be a semisimple Lie subalgebra.
\begin{enumerate}[(a)]
\item\label{32a}
Any semisimple Lie subalgebra $\ss'\subset\dd'$ admits a lift $\ss\subset\aa_0$.
\item\label{32b} For any $e\in\QQ_+$, there is an $\ss$-module splitting of $\pi_e$ in \eqref{64}.
\end{enumerate}
\end{lem}

\begin{proof}\pushQED{\qed}\
\begin{asparaenum}[(a)]
\item By \eqref{23}, $\ss'=[\ss',\ss']\subset[\dd',\dd']=[\dd,\dd]'\subset\aa'$.
Denoting by $r(-)$ the solvable radical, $\ss'\cap r(\aa')$ is a solvable ideal in $\ss'$ and hence zero.
It follows that $\ss'\subset\aa'/r(\aa')$.
By Weyl's complete reducibility theorem, any finite dimensional module over semisimple Lie algebra is semisimple. 
In particular, any epimorphism of such modules splits.
Thus, $\aa_0/r(\aa_0)\onto\aa'/r(\aa')$ splits by semisimplicity of $\aa_0/r(\aa_0)$.
The image of $\ss'$ under this splitting composed with a Levi splitting $\aa_0/r(\aa_0)\into\aa_0$ is the desired lift $\ss$.
\item By part~\eqref{32a}, $\ss\subset\aa_0$ commutes with $\chi$ and hence $\pi_e$ is an epimorphism of $\ss$-modules.
Since $\ss$ is semisimple the desired splitting exists due to Weyl's complete reducibility theorem.\qedhere
\end{asparaenum}
\end{proof}

Speaking in terms of variables, Lemma~\ref{32}.\eqref{32b} associates to each $x_i$ a new variable $x_i'\in \mm_{w_i}$ namely the image of $\bar x_i$ under the splitting map. 
By definition of $\mm_{w_i}$ (see \eqref{15} and Notation~\ref{25}), we have $\chi(x_i')=w_ix_i'$ for all $i=1,\dots,n+1$.
Thus
\[
\chi=\sum_{i=1}^{n+1}w_ix_i'\frac{\p}{\p x_i'}
\]
retains the form of \eqref{14} in the coordinates $x_1',\dots,x_{n+1}'$ where $\deg(x_i')=w_i$ for $i=1,\dots,n+1$.
In other words, homogeneity is unaffected by the coordinate change.
Considering the isomorphism $\pi_{w_1}$ in \eqref{35} as an equality we have $x_i=x_i'$ for all $x_i\in\mm'$. 
In particular,
\[
\ideal{x_1',\dots,x_k'}_\CC=\mm'=\ideal{x_1,\dots,x_k}_\CC.
\]
In other words, the first $k$ variables generate the space of lowest weight variables in both coordinate systems.
However, an additional feature of the coordinates $x_1',\dots,x_{n+1}'$ is that $\ss$ acts on the coordinate  space $\ideal{x_1',\dots,x_{n+1}'}_\CC$.
This is what we mean by saying that $\ss$ acts linearly in terms of the coordinates $x_1',\dots,x_{n+1}'$.

\begin{prp}\label{33}
Let $D$ be a quasihomogeneous divisor germ satisfying condition~\eqref{18} of being unsuspended and condition~\eqref{36} of admitting non-positive $\chi$-generators of logarithmic derivations for any choice of logarithmic Euler derivation $\chi$ along $D$.
Then $\aa'_\chi$ is solvable for a suitable choice of $\chi$.
\end{prp}

\begin{proof}
By quasihomogeneity of $D$ there exists a logarithmic Euler derivation $\chi$ along $D$.
As before, we write $\aa'=\aa'_\chi$ for the representation of $\aa=\aa_\chi$ on $\mm'=\mm_\chi'$.
Suppose $\aa'$ is not solvable.
We shall describe a procedure that gradually modifies $\chi$ to decrease $\dim_\CC(\mm')$ until $\aa'$ becomes solvable.
Being not solvable, $\aa'$ must contain a semisimple Lie subalgebra $\ss'$. 
By Lemma~\ref{32}, $\ss'$ lifts to a semisimple Lie subalgebra $\ss\subset\aa_0$ acting linearly in terms of suitable coordinates $x_1,\dots,x_{n+1}$.
That is, $\ss$ acts on $\bar\mm=\ideal{x_1,\dots,x_{n+1}}_\CC$ and it acts faithfully on its subspace $\mm'=\ideal{x_1,\dots,x_k}_\CC$.
By the structure theory of complex semisimple Lie algebras, $\ss$ contains a Lie subalgebra isomorphic to $\sl_2(\CC)$.
We may assume that $\ss\cong\sl_2(\CC)$ which is then generated by an $\sl_2$-triple.
By the representation theory of $\sl_2$, the semisimple generator $\sigma$ in this triple has both positive and negative integer eigenvalues on the faithful $\ss$-module $\mm'$.
Since $\ss\subset\aa_0\subset\AA=\Der_X(-\log f)$, $\sigma$ commutes with $\chi$ and can be seen as a logarithmic derivation along $D$ annihilating $f$ (see \eqref{55}, \eqref{19}, \eqref{24} and Notation~\ref{25}).
After a linear change of coordinates, we may assume that $x_1,\dots,x_{n+1}$ are eigenvectors of both $\chi$ and $\sigma$. 
For small $\epsilon\in\QQ_+$,
\[
\tilde\chi:=\chi+\epsilon\sigma=\sum_{i=1}^{n+1}\tilde w_ix_i\frac\p{\p x_i}
\]
then takes the form of \eqref{14} for some $\tilde w_i\in\QQ_+$ such that $\tilde w_i<\tilde w_j$ for any $i\le k$ and $k+1\le j$.
That is, $\tilde\chi$ is another logarithmic Euler derivation along $D$ with 
\[
0\ne\mm_{\tilde\chi}'\subsetneq\mm_\chi'
\]
by \eqref{65}.
As illustrated in Figure~\ref{38}, this condition can be achieved also by a less restrictive choice of $\epsilon$ where $\tilde w_i\ge\tilde w_j$ for some $i\le k$ and $k+1\le j$.
Since $\aa'$ is solvable if $k=\dim_\CC(\mm')=1$, repeatedly passing from $\chi$ to $\tilde\chi$ results in a logarithmic Euler derivation $\chi$ for which $\aa'_\chi$ is solvable.
\end{proof}

\begin{figure}[h]
\begin{tikzpicture}[scale=0.05]
\draw[->] (-10,0) -- (-10,80);
\draw[->] (-10,0) -- (200,0);

\foreach \i in {1,...,19} \draw (\i*10,1)--(\i*10,-1) node[below] {\small$x_{\i}$};

\draw[dashed] (10,15)--(40,15);
\draw[dashed] (70,20)--(90,20);
\draw[dashed] (80,30)--(100,30);
\draw[dashed] (30,35)--(150,35);

\draw (-11,25)--(-9,25) node[left] {\small$w_1$};
\draw (-11,15)--(-9,15) node[right] {\small$\tilde w_1$};
\draw (10,25)--(80,25);
\draw[dotted] (10,15)--(30,35);
\foreach \i in {0,...,2} \draw (10+\i*10,15+\i*10) node[shape=circle,draw,fill=white,scale=0.5] {};
\draw[dotted] (40,15)--(60,35);
\foreach \i in {0,...,2} \draw (40+\i*10,15+\i*10) node[shape=circle,draw,fill=white,scale=0.5] {};
\draw[dotted] (70,20)--(80,30);
\foreach \i in {0,...,1} \draw (70+\i*10,20+\i*10) node[shape=circle,draw,fill=white,scale=0.5] {};
\draw (10,25)--(50,25);
\foreach \i in {0,...,7} \draw (10+\i*10,25) node[shape=circle,draw,fill=black,scale=0.3] {};

\draw (-11,45)--(-9,45) node[left] {\small$w_2$};
\draw (-11,20)--(-9,20) node[right] {\small$\tilde w_2$};
\draw (90,45)--(190,45);
\draw[dotted] (90,20)--(140,70);
\foreach \i in {0,...,5} \draw (90+\i*10,20+\i*10) node[shape=circle,draw,fill=white,scale=0.5] {};
\draw[dotted] (150,35)--(170,55);
\foreach \i in {0,...,2} \draw (150+\i*10,35+\i*10) node[shape=circle,draw,fill=white,scale=0.5] {};
\foreach \i in {0,...,1} \draw (180+\i*10,45) node[shape=circle,draw,fill=white,scale=0.5] {};
\foreach \i in {0,...,10} \draw (90+\i*10,45) node[shape=circle,draw,fill=black,scale=0.3] {};

\draw (-11,30)--(-9,30) node[right] {\small$\tilde w_3$};
\draw (-11,35)--(-9,35) node[right] {\small$\tilde w_4$};
\draw (-5,45) node {\small$\vdots$};
\draw (-15,55) node {\small$\vdots$};

\end{tikzpicture}
\caption{An example modification of weights as in Proposition~\ref{33}.}\label{38}
\end{figure}
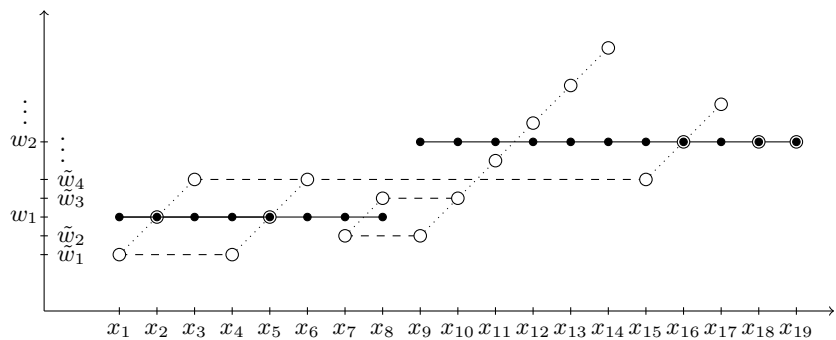

\section{Proof of the main result}

\begin{proof}[Proof of Theorem~\ref{0}]
Let $D$ be a quasihomogeneous free divisor germ that is normal crossing in codimension one.
By Theorem~\ref{39} it suffices to show that all irreducible components of $D$ are smooth.
Due to Proposition~\ref{42} and Lemma~\ref{49}.\eqref{49a}, these irreducible components are again quasihomogeneous, free and normal crossing in codimension one.
By Lemmas~\ref{47} and \ref{49}.\eqref{49b} the latter properties are also invariant under suspension.
We may therefore assume that $D$ is irreducible and unsuspended.
Hence Proposition~\ref{20} applies to show that, for any logarithmic Euler derivation $\chi$ along $D$, the module $\Der_X(-\log D)$ admits a $\chi$-homogeneous basis whose elements have non-positive $\chi$-degree.
Recall the complex Lie algebra $\aa'=\aa_\chi'$ defined by setting $\fl:=\aa$ in Notation~\ref{3} with $\aa=\aa_\chi$ as defined in \eqref{62}.
By Proposition~\ref{33} $\aa'$ is solvable for a suitable $\chi$.
Thus, Propositions~\ref{31} yields smoothness of $D$ as desired.
In fact, after the preceding reductions, we must have $D=\{0\}\subset(\CC,0)$ by Remark~\ref{45} since $D$ is both unsuspended and smooth.
\end{proof}

\bibliographystyle{amsalpha}
\bibliography{fdnc}
\end{document}